\documentclass{amsart}
\usepackage{amsfonts}
\usepackage{color}
\usepackage{graphicx}

\newtheorem{thm}{Theorem}[section]
\newtheorem{cor}[thm]{Corollary}
\newtheorem{lema}[thm]{Lemma}
\newtheorem{prop}[thm]{Proposition}
\theoremstyle{definition}
\newtheorem{defn}[thm]{Definition}
\theoremstyle{remark}

\newtheorem{rem}[thm]{Remark}
\numberwithin{equation}{section}
\newcommand{\R}{\mathbb R}
\newcommand{\ve}{\varepsilon}
\newcommand{\N}{\mathbb N}
\newcommand{\Z}{\mathbb Z}

\newcommand{\lam}{\lambda}

\newcommand{\cd}{\rightharpoonup}

\newcommand{\cf}{\rightarrow}

\begin{document}
\title[Eigenvalue homogenization]{Convergence rate for quasilinear eigenvalue homogenization}
\author[J Fern\'andez Bonder, J P Pinasco, A M Salort]{Juli\'an Fern\'andez Bonder, Juan P. Pinasco, Ariel M. Salort }
\address{Departamento de Matem\'atica
 \hfill\break \indent FCEN - Universidad de Buenos Aires and
 \hfill\break \indent   IMAS - CONICET.
\hfill\break \indent Ciudad Universitaria, Pabell\'on I \hfill\break \indent   (1428)
Av. Cantilo s/n. \hfill\break \indent Buenos Aires, Argentina.}
\email[J. Fernandez Bonder]{jfbonder@dm.uba.ar}
\urladdr[J. Fernandez Bonder]{http://mate.dm.uba.ar/~jfbonder}
\email[J.P. Pinasco]{jpinasco@dm.uba.ar}
\urladdr[J.P. Pinasco]{http://mate.dm.uba.ar/~jpinasco}
\email[A.M. Salort]{asalort@dm.uba.ar}




\begin{abstract}
In this work we study the homogenization problem for nonlinear eigenvalues of
quasilinear elliptic operators. We obtain an explicit order of convergence in $k$ and
in $\ve$ for the (variational) eigenvalues.
\end{abstract}

\maketitle

\section{Introduction}

In this paper we study the asymptotic behavior (as $\ve\to 0$) of the eigenvalues of
the following problems
\begin{equation} \label{pro1}
\begin{cases}
  -div(a_\ve(x,\nabla u^\ve))=\lam^\ve \rho_\ve |u^\ve|^{p-2}u^\ve &\quad \textrm{ in } \Omega \\
   u^\ve =0 &\quad \textrm{ on } \partial \Omega, \\
 \end{cases}
\end{equation}
where $\Omega \subset \R^N$ is a bounded domain,  $\ve$ is a positive real number, and
$\lam^\ve$ is the eigenvalue parameter.

The weight functions $\rho_\ve(x)$ are assumed to be positive and uniformly bounded
away from zero and infinity
\begin{equation}\label{cota.rho}
0<\rho^-\le \rho_\ve(x)\le \rho^+<\infty
\end{equation}
and the family of operators $a_\ve(x,\xi)$ have precise
hypotheses that are stated below, but the prototypical example is
\begin{equation} \label{tipico}
a_\ve(x,\nabla u^\ve) = A^\ve(x)|\nabla u^\ve|^{p-2}\nabla u^\ve,
\end{equation}
with $1<p<+\infty$, and $A^\ve(x)$ is a family of uniformly elliptic matrices (both in
$x\in\Omega$ and in $\ve>0$).

The study of this type of problems have a long history due to its relevance in different fields of applications. The problem of finding the asymptotic behavior of the eigenvalues of \eqref{pro1} is an important part of what is called {\em Homogenization Theory}. Homogenization Theory is applied in composite materials in which the physical parameters such as conductivity and elasticity are oscillating.  Homogenization Theory try to get a good approximation of the macroscopic behavior of the heterogeneous material by letting the parameter $\ve \cf 0$. The main references for the homogenization theory of periodic structures are the books by Bensoussan-Lions-Papanicolaou \cite{BLP78}, Sanchez--Palencia \cite{SP70}, Ole{\u\i}nik-Shamaev-Yosifian \cite{Ol}  among others.

In the linear setting (i.e., $a_\ve(x,\xi)$ as in \eqref{tipico} with $p=2$) this problem is well understood. It is known that, up to a subsequence, there exists a limit operator $a_h(x,\xi)=A^h(x)\xi$ and a limit function $\bar\rho$ such that the spectrum of \eqref{pro1} converges to that of the limit problem.
\begin{equation} \label{limit.prob}
\begin{cases}
  -div(a_h(x,\nabla u))=\lam \bar\rho |u|^{p-2}u &\quad \textrm{ in } \Omega \\
   u =0 &\quad \textrm{ on } \partial \Omega, \\
 \end{cases}
\end{equation}

Let us recall that the convergence of eigenvalues in the multidimensional linear case was studied in 1976 by Boccardo and Marcellini \cite{BM76} for general bounded matrices. Kesavan \cite{Kes1} studied the problem in an periodic frame.
This is an important case of homogenization, i.e. when $\rho_\ve(x) = \rho(x/\ve)$ and $A_\ve(x) = A(x/\ve)$ where $\rho(x)$ and $A(x)$ are $Q-$periodic functions, $Q$ being the unit cube in $\R^N$ and in this case, the limit problem can be fully characterized and so the entire sequence $\ve\to0$ is convergent. See \cite{Kes1, Kes2}.

In the general nonlinear setting, Baffico, Conca and Donato \cite{Con},
relying on the $G-$convergence results of Chiad\'o Piat, Dal Maso and Defranceschi
\cite{Chi1} for monotone operators, study the convergence problem of the principal
eigenvalue of \eqref{pro1}. The concept of $G-$convergence of linear elliptic second
order operators was introduced by Spagnolo in \cite{SP68}. See Section 2 for the
precise definitions.

The convergence problem for higher (variational) nonlinear eigenvalues was addressed
by T. Champion and L. De Pascale in \cite{ChP07} where by means of the $G-$conver-gence
methods they are able to show convergence of the (variational) eigenvalues of
\eqref{pro1} to those of the limit equation \eqref{limit.prob}.

Now, we turn our attention to the order of convergence of the eigenvalues that is the focus of this work. To this end, we restrict ourselves to the problems
\begin{equation} \label{pro2}
\begin{cases}
  -div(a(x, \nabla u^\ve) \nabla u^\ve)=\lam^\ve \rho_\ve |u^\ve|^{p-2}u^\ve &\quad \textrm{ in } \Omega \\
   u^\ve =0 &\quad \textrm{ on } \partial \Omega,
 \end{cases}
\end{equation}
where the family of weight functions $\rho_\ve$ are given in terms of a single bounded $Q-$periodic function $\rho$ in the form $\rho_\ve(x) := \rho(x/\ve)$, $Q$ being the unit cube of $\R^N$.

The limit problem is then given by
\begin{equation} \label{pro3}
\begin{cases}
  -div(a(x, \nabla u) \nabla u)=\lam \bar{\rho} |u|^{p-2}u &\quad \textrm{ in } \Omega \\
   u =0 &\quad \textrm{ on } \partial \Omega,
 \end{cases}
\end{equation}
where $\bar\rho$ is the average of $\rho$ in $Q$, i.e. $\bar\rho = \int_Q \rho(x)\, dx$.

The first estimate of the order of convergence of the eigenvalues, for the linear case, can be found in Chapter III, section 2 of \cite{Ol}. By estimating the eigenvalues of the inverse operator, which is compact,  and using tools from functional analysis in Hilbert spaces, it is prove that
$$
|\lam_k^{\ve}-\lam_k| \le \frac{C\lam_k^{\ve}(\lam_k)^2 }{1-\lam_k\beta_k^{\ve}} \; \ve^{\frac12}.
$$
Here, $C$ is a positive constant, and $\beta_{\ve}^k$ satisfies
$$
0\le \beta_{\ve}^k < \lam_k^{-1},
$$
with
$$
\lim_{\ve \to 0}\beta_{\ve}^k=0
$$
for each $k\ge 1$.

The problem, again in the linear setting and in dimension $N=1$, with $a=1$, was recently studied by Castro and Zuazua in \cite{Zua2, Zua1}. In those articles the authors, using the so-called WKB method which relays on asymptotic expansions of the solutions of the problem, and the explicit knowledge of the eigenfunctions and eigenvalues of the constant coefficient limit problem, proved
$$
|\sqrt{\lam_k^\ve} - \sqrt{\lam_k}|\le C k^2 \ve,
$$
equivalently, since the Weyl's formula implies that $\lam_k^\ve \sim k^2$,
$$
|\lam_k^\ve - \lam_k| \le C k^3\ve
$$
and they also presented a variety of results on correctors for the eigenfunction approximation. Let us mention that their method needs higher regularity on the weight $\rho$, which must belong at least to $C^2(\Omega)$ and that the bound holds for $k\sim \ve^{-1}$.

In the linear problem, in any space dimension, Kenig, Lin and Shen \cite{Kenig}  (allowing an $\ve$ dependance in the diffusion matrix of the elliptic operator) proved that for Lipschitz domains $\Omega$ one has
$$
|\lam_k^\ve - \lam_k| \le C \ve |\log(\ve)|^{\frac12 + \sigma}
$$
for any $\sigma>0$, $C$ depending on $k$ and $\sigma$.

Moreover, the authors show that if the domain $\Omega$ is more regular ($C^{1,1}$ is enough) they can get rid of the logarithmic term in the above estimate. However, no explicit dependance of $C$ on $k$ is obtained in that work.

Later on, in \cite{Kenig2} the authors obtain the precise dependence on $k$ showing that
$$
|\lam_k^\ve - \lam_k| \le C k^{\frac{3}{N}} \ve |\log(\ve)|^{\frac12 + \sigma}
$$
for any $\sigma>0$, $C$ depending on  $\sigma$. Again, when the domain $\Omega$ is smooth, the logarithmic term can be removed.

In this paper  we analyze the order of convergence of eigenvalues of \eqref{pro2} to
the ones of \eqref{pro3} and we prove that,
$$
|\lam_k^\ve - \lam_k|\le C k^{\frac{p+1}{N}} \ve
$$
with $C$ independent of $k$ and $\ve$. In this result, by $\lam_k^\ve$ and $\lam_k$ we refer to the variational eigenvalues of problems \eqref{pro2} and \eqref{pro3} respectively.

\medskip

Some remarks are in order:
\begin{enumerate}
\item Classical Weyl's estimates on the eigenvalues of second order
    $N$-dimensional problems show that $\lam_k$ and $\lam_k^\ve$ behaves like
    $ck^{\frac{2}{N}}$, with $c$ depending only on the coefficients of the
    operator and $N$. Hence, the order of growth of the right-hand side in the
    estimate of \cite{Ol} is
$$
\frac{\lam_k^{\ve}(\lam_k)^2  \ve^{\frac12}}{1-\lam_k\beta_k^{\ve}} \sim \frac{k^{\frac{6}{N}} \ve^{\frac{1}{2}}}{1-\lam_k\beta_k^{\ve}}  \le C k^{\frac{6}{N}} \ve^{\frac12}.
$$
Moreover, the constant involved in their bound is unknown.

\item If we specialize our result to the one dimensional linear case, we recover the estimate obtained in \cite{Zua1}. Moreover, we are considering more general weights $\rho$ since very low regularity is needed and the estimate is valid for any $k$. On the other hand, no corrector results are presented here.

\item In our result no regularity assumptions on $\Omega$ are needed. Any bounded open set will do.

\item The constant $C$ entering in our estimate, can be computed {\em explicitly} in terms of the weight $\rho$, the diffusion operator $a(x,\xi)$ and the Poincar\'e constant in the unit cube of $\R^N$ (see Theorem \ref{explicit}).
\end{enumerate}

In the one dimensional problem, we can provide a extremely elementary proof of Theorems \ref{rate} and \ref{explicit}, generalizing the estimates obtained in \cite{Zua1}. Moreover, in   this case an oscillating  coefficient can
be allowed.

Let us mention that in the nonlinear case considered in this paper there are no quantitative  estimates on the convergence of the eigenfunctions. However, it is possible to prove in
one spatial dimension that the zeros of the eigenfunctions converge to the zeros of the ones of the limit problem and we also find an explicit estimate of the rate of convergence of the nodal domain.

\subsection*{Organization of the paper}

The rest of the paper is organized as follows: In Section 2, we collect some
preliminary results on monotone operators that are needed in order to deal with
\eqref{pro1} and some facts about  the eigenvalue problem. We also recall the definition and some properties of $G-$convergence.
In Section 3 we prove the convergence of the $k$th--variational eigenvalue of problem
\eqref{pro2} to the $k$th--variational eigenvalue of the limit problem \eqref{pro3},
and we study the rates of convergence. In Section 4 we study the one dimensional
problem, and we show in Section 5 that in this case, an oscillating  coefficient can
be allowed. In Section 6 we deal with the zeros of eigenfunctions, and we close the paper in Section 7 with some numerical results.


\section{Preliminary results}

In this section we review some results gathered from the literature, enabling us to clearly state our results and making the paper self-contained.

\subsection{Monotone operators}
We consider $\mathcal{A}\colon W_0^{1,p}(\Omega)\to W^{-1,p'}(\Omega)$ given by
$$
\mathcal{A} u := -div(a(x,\nabla u)),
$$
where $a\colon \Omega\times \R^N\to \R^N$ satisfies,  for every $\xi\in\R^N$ and  a.e. $x\in \Omega$, the following conditions:

\begin{enumerate}
\item[(H0)] {\em measurability:} $a(\cdot,\cdot)$ is a Carath\'eodory function, i.e. $a(x,\cdot)$ is continuous a.e. $x\in \Omega$,  and $a(\cdot,\xi)$ is measurable for every $\xi\in\R^N$.

\item[(H1)] {\em monotonicity:} $0\le (a(x,\xi_1)-a(x,\xi_2))(\xi_1-\xi_2)$.

\item[(H2)] {\em coercivity:} $\alpha |\xi|^p \le a(x,\xi) \xi$.

\item[(H3)] {\em continuity:} $a(x,\xi)\le \beta|\xi|^{p-1}$.

\item[(H4)] {\em $p-$homogeneity:} $a(x,t\xi)=t^{p-1} a(x,\xi)$ for every $t>0$.

\item[(H5)] {\em oddness:} $a(x,-\xi) = -a(x,\xi)$.
\end{enumerate}

Let us introduce $\Psi(x,\xi_1, \xi_2)=a(x,\xi_1) \xi_1 +a(x,\xi_2) \xi_2$ for all $\xi_1, \xi_2 \in \R^N$, and all $x\in \Omega$; and let $\delta=min\{p/2, (p-1)\}$.

\begin{enumerate}
\item[(H6)] {\em equi-continuity:}
$$
|a(x,\xi_1) -a(x,\xi_2)| \le c \Psi(x,\xi_1, \xi_2)^{(p-1-\delta)/p}(a(x,\xi_1) -a(x,\xi_2)) (\xi_1-\xi_2)^{\delta/p}
$$

\item[(H7)] {\em cyclical monotonicity:} $\sum_{i=1}^k  a(x,\xi_i) (\xi_{i+1}-\xi_i) \le 0$, for all $k\ge 1$, and $\xi_1,\ldots, \xi_{k+1}$, with $\xi_1=\xi_{k+1}$.

\item[(H8)] {\em strict monotonicity:} let $\gamma = \max(2,p)$, then
$$
\alpha |\xi_1-\xi_2|^{\gamma}\Psi(x,\xi_1,\xi_2)^{1-(\gamma/p)}\le (a(x,\xi_1)-a(x,\xi_2)) (\xi_1-\xi_2).
$$
\end{enumerate}

See \cite{Con}, Section 3.4 where a detailed discussion on the relation and implications of every condition (H0)--(H8) is given.

In particular, under these conditions, we have the following Proposition:

\begin{prop}[\cite{Con}, Lemma 3.3]\label{potential.f}
Given $a(x,\xi)$ satisfying {\em (H0)--(H8)} there exists a unique Carath\'eodory function $\Phi$ which is even, $p-$homogeneous strictly convex and differentiable in the variable $\xi$ satisfying
\begin{equation}\label{cont.coer.phi}
\alpha |\xi|^p \le \Phi(x,\xi)\le \beta |\xi|^p
\end{equation}
for all $\xi\in \R^N$ a.e. $x\in\Omega$ such that
$$
\nabla_{\xi} \Phi(x,\xi)= p a(x,\xi)
$$
and normalized such that $\Phi(x,0)=0$.
\end{prop}

\subsection{The nonlinear eigenvalue problem}\label{autovalores}
In this section we review some properties of the spectrum of \eqref{pro1} for fixed
$\ve$. That is, we study
\begin{equation}\label{autov.eq}
\begin{cases}
  -div(a(x,\nabla u ))=\lam  \rho  |u|^{p-2}u  &\quad \textrm{ in } \Omega \\
   u  =0 &\quad \textrm{ on } \partial \Omega, \\
\end{cases}
\end{equation}
where $0<\rho_-\le \rho(x)\le \rho_+$ and $0<\alpha\le a(x)\le \beta$ for some
constants $\rho_-, \rho_+, \alpha$ and $\beta$.

We denote by
$$
\Sigma := \{\lam\in\R\colon \mbox{there exists } u\in W^{1,p}_0(\Omega), \mbox{ nontrivial solution to \eqref{autov.eq}}\}
$$
the spectrum of \eqref{autov.eq}.

By means of the critical point theory of Ljusternik--Schnirelmann it is straightforward
to obtain a discrete sequence of variational eigenvalues $\{\lam_k\}_{k\in\N}$ tending
to $+\infty$. We denote by $\Sigma_{var}$ sequence of variational eigenvalues.

The $k$th--variational eigenvalue is given by
\begin{equation} \label{autov.varia}
\lam_k = \inf_{C\in \Gamma_k} \sup_{v \in C} \frac{\int_{\Omega} a(x, \nabla v) |\nabla v |^2 }{\int_{\Omega} \rho(x) |v|^p}
\end{equation}
where
$$\Gamma_k=\{C\subset W^{1,p}_0(\Omega) : C \textrm{ compact, } C=-C, \,  \, \gamma(C)\geq k\}$$
and $\gamma(C)$ is the Kranoselskii genus, see \cite{Rab2} for the definition and
properties of $\gamma$.

For the one dimensional $p-$Laplace operator in $J=(0,\ell)$,
\begin{equation}\label{unaec}
-(|w'|^{p -2} w')' =  \mu  |w|^{p -2}w
\end{equation}
with zero Dirichlet boundary conditions $w(0)=w(\ell)=0$, we have
\begin{equation}\label{formuvar}
\mu_{k} = \inf _{C\in \mathcal{C}_k} \sup _{v\in C} \frac{\int_J |v'|^p\, dx}{\int_J |v|^p\, dx},
\end{equation}
with  $v\in W_0^{1,p}(J)$.

Here, all the eigenvalues and eigenfunctions can be found explicitly:

\begin{thm}[Del Pino, Drabek and Manasevich, \cite{DDM}]\label{draman}
The eigenvalues $\mu_{k}$ and eigenfunctions $w_{k}$ of equation \eqref{unaec} on the interval $J$ are given by
$$
\mu_{k} = \frac{\pi_p^p k^p}{\ell^p},
$$
$$
w_{k}(x) = \sin_p(\pi_p kx/\ell).
$$
\end{thm}

\begin{rem} It was proved in \cite{DM} that they coincide with the variational eigenvalues given by equation \eqref{formuvar}. However, let us observe that the notation is different in both papers.
\end{rem}

The function $\sin_p(x)$ is the solution of the initial value problem
\begin{equation*}
 \begin{cases}
-(|v'|^{p -2} v')' =  |v|^{p -2}v\\
\; v(0)=0, \qquad v'(0)=1,
 \end{cases}
\end{equation*}
and is defined implicitly as
$$
x =  \int_0^{\sin_p(x)} \Big(\frac{p-1}{1-t^p}\Big)^{1/p} dt.
$$
Moreover, its first zero is $\pi_p$, given by
$$
\pi_p = 2 \int_0^1 \Big(\frac{p-1}{1-t^p}\Big)^{1/p} dt.
$$

In \cite{ACM}, problem \eqref{autov.eq} in one space dimension with $a(x,u') = |u'|^{p-2}u'$ is studied and, among other things, it is proved that any eigenfunction associated to $\lambda_k$ has exactly $k$ nodal domains. As a consequence of this fact follows the simplicity of every variational eigenvalue.

The exact same proof of \cite{ACM} works for genera $a(x,u')$, and so we obtain the following:
\begin{thm} \label{teo.zero}
Every eigenfunction of \eqref{autov.eq} corresponding to the $k$th--eigenvalue $\lambda_k$ has exactly $k-1$ zeroes. Moreover, for every $k$, $\lam_k$ is simple, consequently the eigenvalues are ordered as $0<\lam_1<\lam_2 < \cdots < \lam_k \nearrow +\infty$.
\end{thm}

Now, using the same ideas as in \cite{FBP} is easy to prove, for the one dimensional problem, that the spectrum of \eqref{autov.eq} coincides with the variational spectrum. In fact, we have:
\begin{thm} \label{teo.espec}
$\Sigma=\Sigma_{var}$, i.e.,  every solution of problem \eqref{unaec} is given by \eqref{formuvar}.
\end{thm}

\begin{proof}
The proof of this theorem is completely analogous to that of Theorem 1.1 in \cite{FBP}.
\end{proof}

\subsection{Definition of $G-$convergence}\label{gammaconv}

We include the definition of $G-$convergence and the main results in \cite{Con, Chi1} for a sake of completeness, although we will need these facts only in Section 5.

\begin{defn}
We say that the family of operators $\mathcal{A}_\ve u := -div(a_\ve(x,\nabla u))$ $G-$converges to $\mathcal{A}u:=-div(a(x,\nabla u))$ if for every $f\in W^{-1,p'}(\Omega)$  and for every $f_\ve$ strongly convergent to $f$ in $W^{-1,p'}(\Omega)$, the solutions $u^\ve$ of the problem
\begin{equation*}
\begin{cases}
-div(a_\ve(x,\nabla u^\ve))=f_\ve &\textrm{ in } \Omega \\
u^\ve=0 & \textrm{ on } \partial \Omega
\end{cases}
\end{equation*}
satisfy the following conditions
\begin{align*}
 u^\ve \cd u  &\qquad \mbox{ weakly in }  W^{1,p}_0(\Omega), \\
 a_\ve(x, \nabla u^\ve) \cd a(x,\nabla u)  &\qquad \mbox{ weakly in }  (L^{p}(\Omega))^N,
\end{align*}
where $u$ is the solution to the equation
\begin{equation*}
\begin{cases}
-div(a(x,\nabla u))=f & \textrm{ in } \Omega  \\
u=0 &\textrm{ on } \partial \Omega.
\end{cases}
\end{equation*}
\end{defn}

For instance, in the linear periodic case,  the family $-div(A(x/\ve)\nabla u)$ $G$-converges to a limit operator $-div(A^*\nabla u)$ where $A^*$ is a constant matrix which can be characterized in terms of $A$ and certain auxiliary functions. See for example \cite{Cio}.

\medskip

It is shown in \cite{Con} that properties (H0)--(H8) are stable under $G-$convergence, i.e.
\begin{thm}[\cite{Con}, Theorem 2.3]
If $\mathcal{A}_\ve u := -div(a_\ve(x,\nabla u))$ $G-$converges to $\mathcal{A}u:=-div(a(x,\nabla u))$ and $a_\ve(x,\xi)$ satisfies {\em (H0)--(H8)}, then $a(x,\xi)$ also satisfies {\em (H0)--(H8)}.
\end{thm}

In the periodic case, i.e.  when $a_\ve(x,\xi)=a(x/\ve, \xi)$, and $a(\cdot, \xi)$ is $Q-$periodic for every $\xi\in \R^N$, one has that $\mathcal{A_\ve}$ $G-$converges to the homogenized operator $\mathcal{A}_h$ given by $\mathcal{A}_h u = -div(a_h(\nabla u))$, where  $a_h:\R^N \cf \R^N$ can be characterized by
\begin{equation}\label{ah}
 a_h(\xi)=\lim_{s\to \infty} \frac{1}{s^N} \int_{Q_s(z_s)} a(x,\nabla \chi^\xi_s + \xi)dx
\end{equation}
 where $\xi\in \R^N$, $Q_s(z_s)$ is the cube of side length $s$ centered at $z_s$  for any family $\{z_s\}_{s>0}$ in $\R^N$, and $\chi^\xi_s$ is the solution of the following auxiliary problem
\begin{equation}
\begin{cases}
 -div(a(x,\nabla \chi^\xi_s + \xi))=0 \quad \textrm{ in } Q_s(z_s) \\
 \chi^\xi_s\in W^{1,p}_0(Q_s(z)),
\end{cases}
\end{equation}
see \cite{DF1992} for the proof.

In the general case, one has the following compactness result due to \cite{Chi1}
\begin{prop}[\cite{Chi1}, Theorem 4.1]\label{Gconv.ve}
Assume that $a_\ve(x,\xi)$ satisfies {\em (H1)--(H3)} then, up to a subsequence, $\mathcal{A}_\ve$ $G-$converges to a maximal monotone operator $\mathcal{A}$ whose coefficient $a(x,\xi)$ also satisfies {\em (H1)--(H3)}
\end{prop}


\section{Rates of convergence}
In this section we prove that $k$th--variational eigenvalue of problem \eqref{pro2} converges to the $k$th--variational eigenvalue of the limit problem \eqref{pro3}.

Moreover, our goal is to estimate the rate of convergence between the eigenvalues. That is, we want to find explicit bounds for the error $|\lambda_k^\ve - \lambda_k|$.

We begin this section by proving some auxiliary results that are essential in the remaining of the paper.

We first prove a couple of lemmas in order to prove Theorem \ref{teo_n_dim} which  is a generalization for $p\neq 2$ of Oleinik's Lemma \cite{Ol}.

We start with an easy Lemma that computes the Poincar\'e constant on the cube of side $\ve$ in terms of the Poincar\'e constant of the unit cube. Although this result is well known and its proof follows directly by a change of variables, we choose to include it for the sake of completeness.

\begin{lema}\label{poincare}
Let $Q$ be the unit cube in $\R^N$ and let $c_q$ be the Poincar\'e constant in the unit cube in $L^q$, $q\ge 1$, i.e.
$$
\| u - (u)_{Q}\|_{L^q(Q)}\le c_q \| \nabla u\|_{L^q(Q)}, \qquad \mbox{for every } u\in W^{1,q}(Q),
$$
where $(u)_{Q}$ is the average of $u$ in $Q$. Then, for every $u\in W^{1,q}(Q_\ve)$ we have
$$
\| u - (u)_{Q_\ve}\|_{L^q(Q_\ve)}\le c_q \ve \| \nabla u\|_{L^q(Q_\ve)},
$$
where $Q_\ve = \ve Q$.
\end{lema}

\begin{proof}
Let $u\in W^{1,q}(Q_\ve)$. We can assume that $(u)_{Q_\ve}=0$. Now, if we denote $u_\ve (y) = u(\ve y)$, we have that $u_\ve\in W^{1,q}(Q)$ and by the change of variables formula, we get
\begin{align*}
\int_{Q_\ve} |u|^q &= \int_{Q} |u_\ve|^q \ve^N  \le c_q^q \ve^N \int_{Q} |\nabla u_\ve|^q = c_q^q \ve^q \int_{Q_\ve} |\nabla u|^q.
\end{align*}
The proof is now complete.
\end{proof}

\begin{rem}
In \cite{AD}, Acosta and Duran show that for convex domains $U$, one has
$$
\|u - (u)_U\|_{L^1(U)}\le \frac{d}{2} \|\nabla u\|_{L^1(U)},
$$
where $d$ is the diameter of $U$. When we apply this result to the unit cube, we get
\begin{equation}\label{c1}
c_1\le \frac{\sqrt{N}}{2}.
\end{equation}
\end{rem}

The next Lemma is the final ingredient in the estimate of Theorem \ref{teo_n_dim}.

\begin{lema} \label{lema.clave}
Let $\Omega\subset \R^N$ be a bounded domain and denote by $Q$ the unit cube in $\R^N$. Let $g\in L^\infty(\R^N)$ be a $Q$-periodic function such that $\bar g=0$. Then the inequality
$$
\left| \int_{\Omega} g(\tfrac{x}{\ve})v \right| \leq \|g\|_{L^\infty(\R^N)} c_1\ve \|\nabla v\|_{L^{1}(\Omega)}
$$
holds for every $v\in W^{1,1}_0(\Omega)$, where $c_1$ is the Poincar\'e constant given in \eqref{c1}.
\end{lema}

\begin{proof}
Denote by $I^\ve$ the set of all $z\in \Z^N$ such that $Q_{z,\ve}\cap \Omega \neq \emptyset$, $Q_{z,\ve}:=\ve(z+Q)$. Given $v\in W^{1,1}_0(\Omega)$ we extended by $0$ outside $\Omega$ and consider the function $\bar v_\ve$ given by the formula
$$
\bar{v}_\ve (x)=\frac{1}{\ve^N}\int_{Q_{z,\ve}} v(y)dy
$$
for $x\in Q_{z,\ve}$. We denote by $\Omega_1 = \cup_{z\in I^\ve} Q_{z,\ve} \supset \Omega$. Then we have
\begin{align} \label{keq0}
\int_{\Omega} g_\ve v =  \int_{\Omega_1} g_\ve (v-\bar{v}_\ve) + \int_{\Omega_1} g_\ve \bar{v}_\ve.
\end{align}

Now, by Lema \ref{poincare} we get
\begin{equation}\label{5.4}
\begin{split}
\|v-\bar{v}_\ve\|_{L^1(\Omega_1)} &= \sum_{z\in I^\ve} \int_{Q_{z,\ve}} |v-\bar{v}_\ve| dx\\
&\leq c_1\ve \sum_{z\in I^{z,\ve}} \int_{Q_{z,\ve}} |\nabla v(x)| dx \\
& \le c_1\ve \|\nabla v\|_{L^1(\Omega)}.
\end{split}
\end{equation}

Finally, since $\bar g = 0$ and since $g$ is $Q-$periodic, we get
\begin{equation}\label{ultima}
\int_{\Omega_1} g_\ve \bar v_\ve = \sum_{z\in I^\ve} \bar v_\ve\mid_{Q_{z,\ve}} \int_{Q_{z,\ve}} g_\ve = 0.
\end{equation}

Now, combining \eqref{5.4} and \eqref{ultima} we can bound \eqref{keq0} by
$$
\Big|\int_{\Omega} g_\ve v\Big| \le \|g\|_{L^\infty(\R^N)} c_1 \ve \|\nabla v\|_{L^{1}(\Omega)}.
$$
This finishes the proof.
\end{proof}

The next Theorem is essential to estimate the rate of convergence of the eigenvalues since it allows us to replace an integral involving a rapidly oscillating function with one that involves its average in the unit cube.

\begin{thm}\label{teo_n_dim}
Let $g\in L^{\infty}(\R^N)$ be a $Q-$periodic function such that $0< g^-\le g\le g^+<\infty$. Then, \begin{align*}
\Big|\int_{\Omega} (g_\ve(x) - \bar g) |u|^p\Big| & \le p c_1 \|g-\bar g\|_{L^\infty(\R^N)} \ve \|u\|_{L^p(\Omega)}^{p-1} \|\nabla u\|_{L^p(\Omega)}\end{align*}
for every $u\in W^{1,p}_0(\Omega)$. Here, $c_1$ is the optimal constant in Poincar\'e's inequality in $L^1(Q)$ given by \eqref{c1}.
\end{thm}

\begin{proof}
Let $\ve>0$ be fixed. Now, denote by $h = g-\bar g$ and so, by Lemma \ref{lema.clave} we obtain
\begin{equation}\label{segunda}
\Big|\int_{\Omega} h_\ve |u|^p\Big| \le \|h\|_{L^\infty(\R^N)} c_1\ve \| \nabla(|u|^p)\|_{L^1(\Omega)}.
\end{equation}
An easy computation shows that
\begin{equation}\label{normas}
\|\nabla(|u|^p)\|_{L^1(\Omega)}\le p \|u\|_{L^p(\Omega)}^{p-1} \|\nabla u\|_{L^p(\Omega)}.
\end{equation}
Finally, combining \eqref{segunda} and \eqref{normas} we obtain the desired result.
\end{proof}

\medskip

Now we are ready to prove the main result of this section.
\begin{thm}\label{rate}
Let $\lam_k^\ve$ be the $k$th--variational eigenvalue associated to equation \eqref{pro2} and let be $\lam_k$ be the $k$th--variational eigenvalue associated to the limit problem \eqref{pro3}. Then
$$
|\lam_k^\ve - \lam_k|\le c_1 p \|\rho-\bar\rho\|_\infty  \big(\frac{\rho^+}{\alpha}\big)^{1/p} \frac{1}{\rho^-}\ve \max\{\lam_k, \lam_k^\ve\}^{1+\frac{1}{p}}.
$$
\end{thm}

\begin{proof}
Let $\delta>0$ and let $G_\delta^k\subset W^{1,p}_0(\Omega)$ be a compact, symmetric set of genus $k$ such that
$$
\lam_k = \inf_{G\in \Gamma_k} \sup_{u \in G} \frac{\int_{\Omega}  \Phi(x,\nabla u)}{\bar{\rho}\int_{\Omega} |u|^p} = \sup_{u \in G_\delta^k} \frac{\int_{\Omega} \Phi(x,\nabla u)}{\bar{\rho}\int_{\Omega} |u|^p} + O(\delta).
$$

We use now the set $G_\delta^k$, which is admissible in the variational characterization of the $k$th--eigenvalue of \eqref{pro2}, in order to found a bound for it as follows,
\begin{align} \label{z1}
\lam_k^\ve  \leq  \sup_{u \in G_\delta^k} \frac{\int_{\Omega} \Phi(x,\nabla u)}{\int_{\Omega}  \rho_\ve |u|^p} = \sup_{u \in G_\delta^k} \frac{\int_{\Omega}  \Phi(x,\nabla u)}{\bar{\rho}\int_\Omega |u|^p }  \; \frac{\bar{\rho}\int_\Omega |u|^p}{\int_{\Omega} \rho_\ve |u|^p}.
\end{align}

To bound $\lam_k^\ve$ we look for bounds of the two quotients in \eqref{z1}. For every function $u\in G_\delta^k$ we have that
\begin{align} \label{z2}
\frac{\int_{\Omega}  \Phi(x,\nabla u)}{\bar{\rho}\int_{\Omega} |u|^p} \leq \sup_{v \in G_\delta^k} \frac{\int_{\Omega}  \Phi(x,\nabla v)}{\bar{\rho}\int_{\Omega} |v|^p} = \lam_k + O(\delta).
\end{align}

Since $u\in G_\delta^k\subset W^{1,p}_0(\Omega)$,  by Theorem \ref{teo_n_dim} we obtain that
\begin{align} \label{z3}
  \frac{\bar{\rho}\int_{\Omega} |u|^p}{\int_{\Omega} \rho_\ve|u|^p} \leq 1+ c_1 p \ve \|\rho-\bar\rho\|_\infty \frac{\|u\|_{L^p(\Omega)}^{p-1}\|\nabla u\|_{L^{p}(\Omega)}}{\int_\Omega \rho_\ve |u|^p}.
\end{align}

Now, by \eqref{cota.rho}, \eqref{cont.coer.phi} together with \eqref{z2}, we have
\begin{equation} \label{z4}
\begin{split}
\frac{\|u\|_{L^p(\Omega)}^{p-1}\|\nabla u\|_{L^{p}(\Omega)}}{\int_\Omega \rho_\ve |u|^p}&\le \frac{\bar\rho^{1/p}}{\rho^-}\Big(\frac{\int_\Omega |\nabla u|^p\, dx}{\int_\Omega \bar\rho |u|^p}\Big)^{1/p}\\
&\le \big(\frac{\bar\rho}{\alpha}\big)^{1/p} \frac{1}{\rho^-}\Big(\frac{\int_\Omega \Phi(x,\nabla u)}{\int_\Omega \bar \rho |u|^p}\Big)^{1/p}\\
&\le  \big(\frac{\bar\rho}{\alpha}\big)^{1/p} \frac{1}{\rho^-}(\lam_k + O(\delta))^{1/p}.
\end{split}
\end{equation}

Then combining \eqref{z1}, \eqref{z2}, \eqref{z3} and \eqref{z4} we find that
$$
\lam_k^\ve \leq (\lam_k + O(\delta)) \left( 1+ c_1 p \|\rho-\bar\rho\|_\infty  \big(\frac{\bar\rho}{\alpha}\big)^{1/p} \frac{1}{\rho^-} \ve (\lam_k + O(\delta))^{1/p}  \right).
$$
Letting $\delta\to 0$ we get
\begin{align} \label{z7}
\lam_k^\ve - \lam_k \leq c_1 p \|\rho-\bar\rho\|_\infty  \big(\frac{\bar\rho}{\alpha}\big)^{1/p} \frac{1}{\rho^-}\ve \lam_k^{1+\frac{1}{p}}.
\end{align}

In a similar way, interchanging the roles of $\lam_k$ and $\lam_k^\ve$, we obtain
\begin{align} \label{z8}
\lam_k - \lam_k^\ve \leq c_1 p \|\rho-\bar\rho\|_\infty  \big(\frac{\rho^+}{\alpha}\big)^{1/p} \frac{1}{\bar\rho}\ve (\lam_k^\ve)^{1+\frac{1}{p}}.
\end{align}

So, from \eqref{z7} and \eqref{z8}, we arrive at
$$
|\lam_k^\ve - \lam_k|\le c_1 p \|\rho-\bar\rho\|_\infty  \big(\frac{\rho^+}{\alpha}\big)^{1/p} \frac{1}{\rho^-}\ve \max\{\lam_k, \lam_k^\ve\}^{1+\frac{1}{p}},
$$
and so the proof is complete.
\end{proof}

It would be desirable to give a precise rate of convergence in terms only con $\ve$, $k$ and the coefficients of the problem. In order to achieve this goal, we need to give explicit bounds on the eigenvalues $\lam_k$ and $\lam_k^\ve$. But this follows by comparison with the $k$th--variational eigenvalue of the $p-$Laplacian, $\mu_k$ and a refinement of the bound on $\mu_k$ proved in \cite{Friedlander}.

In fact, from \eqref{cont.coer.phi} we have
$$
\frac{\alpha}{\bar \rho} \frac{\int_\Omega |\nabla u|^p}{\int_\Omega |u|^p}\le \frac{\int_\Omega \Phi(x,\nabla u)}{\int_\Omega \bar \rho |u|^p}\le \frac{\beta}{\bar \rho} \frac{\int_\Omega |\nabla u|^p}{\int_\Omega |u|^p},
$$
$$
\frac{\alpha}{\rho^+} \frac{\int_\Omega |\nabla u|^p}{\int_\Omega |u|^p}\le \frac{\int_\Omega \Phi(x,\nabla u)}{\int_\Omega  \rho_\ve |u|^p}\le \frac{\beta}{ \rho^-} \frac{\int_\Omega |\nabla u|^p}{\int_\Omega |u|^p},
$$
from where it follows that
$$
\max\{ \lam_k, \lam_k^\ve\}\le \frac{\beta}{\rho^-}\mu_k.
$$
Now, in \cite{Friedlander}, it is shown that
$$
\mu_k \le \tilde\mu_1 \big(\frac{k}{|\Omega|}\big)^{p/N}
$$
where $\tilde\mu_1$ is the first Dirichlet eigenvalue for the $p-$laplacian in the unit cube. Finally, in \cite{FBP} an estimate for $\tilde\mu_1$ by comparing with the first eigenvalue of the {\em pseudo} $p-$laplacian is obtained, namely
$$
\tilde\mu_1 \le \max\{N^{(p-2)/2}, 1\} N \pi_p^p,
$$
where $\pi_p$ is defined in Subsection \ref{autovalores}.

Combining all of these facts, we immediately prove
\begin{thm}\label{explicit}
With the same assumptions and notations as in Theorem \ref{rate} we have
$$
|\lam_k-\lam_k^\ve| \le C \ve k^{\frac{p+1}{N}},
$$
where the constant $C$ is given by
$$
C= \frac{\sqrt{N}}{2} p \|\rho-\bar\rho\|_\infty \Big(\frac{\beta^{p+1}}{\alpha}\Big)^{\frac{1}{p}} \frac{1}{(\rho^-)^2} \Big(\frac{\rho^+}{\rho^-}\Big)^{\frac{1}{p}} \pi_p^{p+1} N^{\frac{p+1}{p}} \max\{N^{\frac{p-2}{2}}, 1\}^{\frac{p+1}{p}}.
$$

\end{thm}

\begin{rem}
As we mentioned in the introduction, in the linear case and in one space dimension Castro and Zuazua \cite{Zua1} prove that, for $k < C \ve^{-1}$,
$$
|\lam_k^\ve - \lam_k|\le C k^3 \ve.
$$
If we specialize our result to this case, we get the same bound. The advantage of our method is that very low regularity on $\rho$ is needed (only $L^\infty$). However, the method in \cite{Zua1}, making use of the linearity of the problem, gives precise information about the behavior of the eigenfunctions $u_k^\ve$.

Moreover, in the one dimensional linear case with diffusion coefficient equal to 1, we can simplify the constant and obtain
$$
|\lam_k^\ve - \lam_k|\le \frac{ \|\rho-\bar\rho\|_\infty }{(\rho^-)^2} \sqrt{\frac{\rho^+}{\rho^-}} (\pi k)^3 \ve
$$
\end{rem}

\begin{rem}
In \cite{Kenig, Kenig2}, Kenig, Lin and Shen studied the linear case in any space dimension (allowing a periodic oscillation diffusion matrix) and prove the bound
$$
|\lam_k^\ve - \lam_k|\le C \ve |\log \ve|^{1+\sigma} \lambda_k^{\frac32}.
$$
for some $\sigma>0$ and $C$ depending on $\sigma$ and $k$ (The authors can get rid off the logarithmic term assuming more regularity on $\Omega$). If we specialize our result to this case, we cannot treat an $\ve$ dependance on the operator, we get the same dependance on $k$ but without any regularity assumption on $\Omega$ we get the optimal dependence on $\ve$.
\end{rem}

\section{The one dimensional problem}

In this section we consider the following one dimensional problems
\begin{equation} \label{Pve}
 \begin{cases}
   -(|u_\ve'|^{p-2}u_\ve')'=\lam^{\ve} \rho(\tfrac{x}{\ve})|u_\ve|^{p-2}u_\ve   \quad \textrm{ in }  I:=(0,1)\\
   u_\ve(0)=u_\ve(1)=0,
 \end{cases}
\end{equation}
and the limit problem
\begin{equation} \label{Plim}
 \begin{cases}
   -(|u_\ve'|^{p-2}u_\ve')'=\lam^{\ve} \bar\rho(x)|u_\ve|^{p-2}u_\ve   \quad \textrm{ in }  I\\
   u_\ve(0)=u_\ve(1)=0,
 \end{cases}
\end{equation}
where $\rho$ is $1-$periodic and $\bar\rho$ is the average of $\rho$ in the unit interval.

In order to prove the rate of convergence, let us assume first that $\bar\rho=0$ and
let $R(x) = \int_0^x \rho(t)\, dt$. Then $R$ is $1-$periodic and if we denote $R_\ve(x)
= R(\tfrac{x}{\ve})$, we get
$$
\ve R'_\ve(x) = \rho(\tfrac{x}{\ve}).
$$
Hence, if $v\in W^{1,1}_0(I)$,
$$
\int_0^1 \rho(\tfrac{x}{\ve}) v(x)\, dx = \ve \int_0^1 R'_\ve(x) v(x)\, dx  = - \ve \int_0^1 R_\ve(x) v'(x)\, dx.
$$
So
$$
\Big|\int_0^1 \rho(\tfrac{x}{\ve}) v(x)\, dx\Big| \le  \ve \|R\|_\infty \int_0^1 |v'(x)|\, dx.
$$
It is easy to see that
$$
 \|R\|_\infty \le \|\max\{\rho, 0\}\|_1 = \frac12 \|\rho\|_1,
$$
since $\bar\rho=0$.

We have proved
\begin{lema}\label{oleinik1d}
Let $v\in W^{1,1}_0(I)$ and $\rho\in L^1(I)$ be such that $\bar\rho := \int_0^1\rho(x)\, dx=0$. Then
$$
\Big|\int_0^1 \rho(\tfrac{x}{\ve}) v(x)\, dx\Big| \le \frac12 \|\rho\|_1 \ve \|v'\|_1.
$$
\end{lema}

Then we get the following corollary
\begin{cor}\label{cota1d}
Let $u\in W^{1,p}_0(I)$ and $\rho\in L^1(I)$. Then
$$
\Big|\int_0^1 (\rho(\tfrac{x}{\ve})-\bar\rho) |u(x)|^p\, dx\Big| \le \frac{p}{2} \ve \|\rho - \bar\rho\|_1 \|u\|_p^{p-1} \|u'\|_p.
$$
\end{cor}
\begin{proof}
It follows just by noticing that
$$
\int_0^1| (|u|^p)' |\, dx = p \int_0^1 |u|^{p-1} |u'|\, dx \le p\|u\|_p^{p-1} \|u'\|_p
$$
and applying Lemma \ref{oleinik1d}.
\end{proof}

Now if we argue exactly as in Theorem \ref{rate} but use Corollary \ref{cota1d} instead of Theorem \ref{teo_n_dim}, we get
$$
|\lam_k^\ve - \lam_k| \le \frac{p}{2} \|\rho-\bar\rho\|_1 \frac{\rho_+^{\frac{1}{p}}}{\rho_-} \ve \max\{\lam_k, \lam_k^\ve\}^{1+\frac{1}{p}}.
$$

The bound for $\lam_k, \lam_k^\ve$ follows directly from Theorem \ref{draman}. In
fact,
$$
\lam_k, \lam_k^\ve \le \frac{1}{\rho_-} \mu_k = \frac{1}{\rho_-} (\pi_p k)^p.
$$

So we have proved:

\begin{thm}\label{teo.1d}
The following estimate holds
$$
|\lam_k^\ve - \lam_k| \le \frac{p}{2} \frac{\|\rho-\bar\rho\|_1}{\rho_-^2} \Big(\frac{\rho_+}{\rho_-}\Big)^{\frac{1}{p}} \ve (\pi_p k)^{p+1}.
$$
\end{thm}

\begin{rem}
If we replace the unit interval $I=(0,1)$ by $I_\ell = (0,\ell)$ by a simple change of
variables, the estimates of Theorem \ref{teo.1d} are modified as
\begin{equation}\label{intervaloL}
|\lam_k^\ve(I_\ell) - \lam_k(I_\ell)| =  \ell^{p} |\lam_k^\ve(I) - \lam_k(I)|.
\end{equation}
\end{rem}

\section{The general equation}

In this section we consider, for the one dimensional problem the case where an
oscillating coefficient in the equation is allowed. i.e., the problem
\begin{equation}\label{coef.osc}
\begin{cases}
-(a(\tfrac{x}{\ve}) |u'|^{p-2}u')' = \lam^\ve \rho(\tfrac{x}{\ve}) |u|^{p-2}u & \text{ in } (0,1)\\
u(0)=u(1)=0
\end{cases}
\end{equation}

We will show that this case can be reduced to Theorem \ref{teo.1d} by a suitable change of variables. In fact, if we define
$$
P_\ve(x) = \int_0^x \frac{1}{a_\ve(s)^{1/(p-1)}} ds= \ve \int_0^{x/\ve} \frac{1}{a(s)^{1/(p-1)}} ds = \ve P(\tfrac{x}{\ve})
$$
and perform the change of variables
$$
(x,u)\to (y, v)
$$
where
$$
y = P_\ve(x) = \ve P(\tfrac{x}{\ve}), \qquad v(y)= u(x).
$$
By simple computations we get
$$
\begin{cases}
 -(|\dot{v}|^{p-2}\dot{v})^{\cdot} =  \lam^\ve Q_\ve(y) |v|^{p-2} v, &  y\in [0,L_\ve]\\
 v(0)=v(L_\ve)=0
 \end{cases}
$$
where
$$\quad\cdot =    d/dy,$$
with
$$
L_\ve = \int_0^1 \frac{1}{a_\ve(s)^{1/(p-1)}} ds \to L = \overline{a^{\frac{-1}{p-1}}},
$$
and
\begin{align*}
Q_\ve(y)& =  a_\ve(x)^{1/(p-1)} \rho_\ve(x) \\
& =   a(P^{-1}(\tfrac{y}{\ve}))^{1/(p-1)}\rho(P^{-1}(\tfrac{y}{\ve})) \\
& =  Q(\tfrac{y}{\ve}).
\end{align*}

Observe that $Q$ is an $L-$periodic function.

Moreover, it is easy to see that
\begin{equation}\label{cota.Lve}
|L_\ve - L| \le \ve L
\end{equation}
and that $L_\ve = L$ if $\ve=1/j$ for some $j\in \N$.

In order to apply Theorem \ref{teo.1d} we need to rescale to the unit interval. So we define
$$
w(z) = v(L_\ve z), \qquad z\in I
$$
and get
$$
\begin{cases}
-(|\dot w|^{p-2}\dot w)^\cdot = L_\ve^p\lam^\ve Q_\ve(L_\ve z) |w|^{p-2}w & \mbox{ in } I\\
w(0)=w(1)=0
\end{cases}
$$
So if we denote $\delta = \ve L/L_\ve$, $\mu^\delta = L_\ve^{p}\lam^\ve$ and $g(z) = Q(Lz)$, we get that $g$ is a $1-$periodic function and that $w$ verifies
$$
\begin{cases}
-(|\dot w|^{p-2}\dot w)^\cdot = \mu^\delta g(\tfrac{z}{\delta}) |w|^{p-2}w & \mbox{ in } I\\
w(0)=w(1)=0
\end{cases}
$$
Now we can apply Theorem \ref{teo.1d} to the eigenvalues $\mu^\delta$ to get
\begin{equation}\label{cota.mu}
|\mu_k^\delta - \mu_k|\leq \frac{p}{2} \frac{\|g-\bar g\|_1}{g_-^2} \Big(\frac{g_+}{g_-}\Big)^{\frac{1}{p}} \delta (\pi_p k)^{p+1}.
\end{equation}
In the case where $\ve = 1/j$ with $j\in \N$ we directly obtain
$$
|\lam_k^\ve - \lam_k| \leq \frac{1}{L^{p}}\frac{p}{2} \frac{\|g-\bar g\|_1}{g_-^2} \Big(\frac{g_+}{g_-}\Big)^{\frac{1}{p}} \ve (\pi_p k)^{p+1}.
$$

In the general case, one has to measure the defect between $L$ and $L_\ve$. So,
\begin{equation}\label{cota.g2}
|\lam_k^\ve - \lam_k| \le \frac{1}{L^{p}}( |\mu_k^\delta - \mu_k| + \lam_k^\ve |L_\ve^{p}-L^{p}|)
\le \frac{1}{L^p} ( |\mu_k^\delta - \mu_k| + \frac{\beta}{\rho_-}\pi_p^p k^p |L_\ve^{p}-L^{p}|).
\end{equation}
From \eqref{cota.Lve} it is easy to see that
$$
|(\tfrac{L_\ve}{L})^{p}-1|\le p(1+\ve)^{p-1}\ve.
$$
so
\begin{equation}\label{cota.g3}
|L_\ve^{p}-L^{p}| = L^{p}|(\tfrac{L_\ve}{L})^{p}-1| \le
p L^p (1+\ve)^{p-1}\ve.
\end{equation}

Finally, using \eqref{cota.mu}, \eqref{cota.g2} and \eqref{cota.g3} we obtain:
\begin{thm}
Let $\lam_k^\ve$ be the $kth-$eigenvalue of
$$
\begin{cases}
-(a(\tfrac{x}{\ve})u')' = \lam^\ve \rho(\tfrac{x}{\ve}) |u|^{p-2}u & \text{in } I=(0,1)\\
u(0)=u(1)=0
\end{cases}
$$
and let $\lam_k$ be the $kth-$eigenvalue of the homogenized limit problem
$$
\begin{cases}
-(a^*_p |u'|^{p-2}u)' = \lam \bar\rho |u|^{p-2}u & \text{in } I\\
u(0)=u(1)=0.
\end{cases}
$$

Then, if $\ve=1/j$ for some $j\in\N$,
$$
|\lam_k^\ve - \lam_k|\le \frac{1}{L^p}\frac{p}{2} \frac{\|g-\bar g\|_1}{g_-^2} \Big(\frac{g_+}{g_-}\Big)^{\frac{1}{p}} \ve (\pi_p k)^{p+1}
$$
and if $\ve\neq1/j$,
$$
|\lam_k^\ve - \lam_k|\le \frac{1}{L^p}\frac{p}{2} \frac{\|g-\bar g\|_1}{g_-^2} \Big(\frac{g_+}{g_-}\Big)^{\frac{1}{p}} \frac{\ve}{1-\ve} (\pi_p k)^{p+1} + \frac{\beta}{\rho_-} p L^p (1+\ve)^{p-1}\ve (\pi_p k)^p.
$$
\end{thm}

\section{Convergence of nodal domains}

In this section we prove the following result about the convergence of the nodal sets and of the zeros of the eigenfunctions.

\begin{thm} \label{teo-medida-dom}
Let $(\lam_k^\ve,u_k^\ve)$ and $(\lam_k,u_k)$ be eigenpairs associated to equations
\eqref{Pve} and \eqref{Plim} respectively. We denote by $\mathcal{N}_k^\ve$ and
$\mathcal{N}_k$ to a nodal domains of $u_k^\ve$ and $u_k$ respectively. Then
$$ |\mathcal{N}_k^\ve| \cf |\mathcal{N}_k| \quad \textrm{ as }\ve \cf 0$$
and we have the estimate
$$ \left| |\mathcal{N}_k^\ve|^{-p} -|\mathcal{N}_k|^{-p} \right| \leq   c\ve(k^{p+1}+1) $$
\end{thm}
\begin{proof}

By using Theorem \ref{explicit}, together with \eqref{intervaloL} and the explicit
form of the eigenvalues of the limit problem we obtain that
\begin{align} \label{w1}
\lam_k^\ve(I)=\lam_1^\ve(\mathcal{N}_k^\ve) \leq \lam_1(\mathcal{N}_k^\ve) + c|\mathcal{N}_k^\ve|^{p-1} \ve \leq \frac{\pi_p^p}{\overline{\rho} |\mathcal{N}_k^\ve|^p} + c \ve.
\end{align}
Also,
\begin{align}\label{w2}
\lam_k^\ve(I)\geq \lam_k(I) - c \ve k^{p+1} = \frac{k^{p}\pi_p^p}{\overline{\rho}} -c \ve k^{p+1}.
\end{align}
As $w_k(x) = \sin_p(k\pi_p x)$ (see Theorem \ref{draman}) has $k$ nodal domain in $I$
we must have $|\mathcal{N}_k|=k^{-1}$. Then by (\ref{w1}) and (\ref{w2}) we get
$$ \frac{\pi_p^p}{\overline{\rho} |\mathcal{N}_k|^p} - c \ve k^{p+1}\leq  \frac{1}{|\mathcal{N}_k^\ve|^p} \frac{\pi_p^p}{\overline{\rho}} +c\ve$$
it follows that
\begin{equation} \label{ex1}
 |\mathcal{N}_k|^{-p} -   |\mathcal{N}_k^\ve|^{-p}  \leq c\ve(k^{p+1}+1).
\end{equation}
Similarly we obtain that
$$
\frac{ \pi_p^p}{\overline{\rho}|\mathcal{N}_k|^p}=\lam_1(\mathcal{N}_k)= \lam_k(I) \geq \lam_k^\ve(I)-c\ve k^{p+1} \geq \lam_1^\ve(\mathcal{N}_k^\ve)-c\ve k^{p+1}
$$
and using again Theorem \ref{explicit} we get
$$
\lam_1^\ve(\mathcal{N}_k^\ve)\geq \lam_1 (\mathcal{N}_k^\ve)-c\ve=\frac{\pi_p^p}{\overline{\rho}|\mathcal{N}_k^\ve|^p} -c\ve
$$
it follows that
\begin{equation} \label{ex2}
|\mathcal{N}_k^\ve|^{-p}  -  |\mathcal{N}_k|^{-p} \leq c\ve(k^{p+1}+1).
\end{equation}
Combining \eqref{ex1} and \eqref{ex2} the result follows.
\end{proof}
Finally, as a corollary of Theorem \ref{teo-medida-dom} we are able to prove the
individual convergence of the zeroes of the eigenfunctions of \eqref{Pve} to those of
the limit problem \eqref{Plim}.
\begin{cor}
Let $(\lam_k^\ve,u_k^\ve)$ and $(\lam_k,u_k)$ be eigenpairs associated to equations
\eqref{Pve} and \eqref{Plim} respectively. Denote $x_j^\ve$ and $x_j$, $0\leq j \leq
k$ its respective zeroes. Then for each $1< j< k$
$$x_j^\ve \cf x_j \quad \textrm{ when }\ve \cf 0$$
and
$$|x_j^\ve - x_j| \leq jc\ve(k^{p+1}+1).$$
In particular $x_0^\ve=x_0=0$ and $x_k^\ve=x_k=1$ by the boundary condition.
\end{cor}

\begin{proof}
With the notation of Theorem \ref{teo-medida-dom} we have that $|\mathcal{N}_k^\ve| \cf |\mathcal{N}_k|$. For the first pair of nodal domains we get
$$|x_1^\ve- x_1|=|x_1^\ve - x_0^\ve - x_1 + x_0|=\left||\mathcal{N}_{k,1}^\ve|-|\mathcal{N}_{k,1}| \right| \leq c\ve(k^{p+1}+1)$$
for the second couple
$$|(x_2^\ve - x_2) -(x_1^\ve -x_1)|=\left||\mathcal{N}_{k,2}^\ve|-|\mathcal{N}_{k,2}| \right| \leq c\ve(k^{p+1}+1)$$
then
$$|x_2^\ve - x_2| \leq c\ve(k^{p+1}+1) +|x_1^\ve -x_1| \leq 2c\ve(k^{p+1}+1).$$

We iterate the reasoning for $j< k$,
$$
|x_{j}^\ve - x_{j}| \leq jc\ve(k^{p+1}+1)
$$
and the proof is complete.
\end{proof}

\section{Some examples and numerical results}

We define the following Pr\"ufer transformation:
\begin{equation} \label{prufer}
\begin{cases}
\left(\frac{\lam r(x)}{p-1} \right)^{1/p}u(x) & = \rho(x)S_p(\varphi(x)), \\
u'(x) & =  \rho(x)C_p(\varphi(x))
\end{cases}
\end{equation}
As in \cite{Pin}, we can see show that $\rho(x)$ and $\varphi(x)$ are continuously
differentiable functions satisfying
\begin{equation} \label{ec_prufer}
\begin{cases}
\varphi'(x)&=\left(\frac{\lam r(x)}{p-1}\right)^{\frac{1}{p}}+\frac{1}{p} \frac{r'(x)}{r(x)} |C_p(\varphi(x))|^{p-2}C_p(\varphi(x))S_p(\varphi(x)) \\
\rho'(x)&=\frac{1}{p} \frac{r'(x)}{r(x)} \rho(x)|S_p(\varphi(x))|^p
\end{cases}
\end{equation}
and we obtain that
$$
u_k(x)=\left(\frac{\lam_k r(x)}{p-1}\right)^{-1/p} \rho_k(x) S_p(\varphi_k(x)), \quad k\geq 1
$$
is a eigenfunction of problem \eqref{unaec} corresponding to $\lam_k$ with zero Dirichlet
boundary conditions.

 We propose the following algorithm to compute the eigenvalues of
problem \eqref{unaec} based in the fact that the eigenfunction associate to $\lam_k$ has
$k$ nodal domain in $I$, so the phase function $\varphi$ must vary between $0$ and
$k\pi_p$. It consists in a shooting method combined with a bisection algorithm (a
Newton-Raphson version can be implemented too).

\begin{align*}
    &\texttt{Let } a < \lam < b \texttt{ and let } \tau \texttt{ be the tolerance}\\
    &\texttt{Solve the ODE } \ref{ec_prufer} \texttt{ and obtain } \varphi_\lam \texttt{ and } \rho_\lam \\
    &\texttt{Let } w(x)=(p-1)^{1/p}\left(\lam r(x)\right)^{-1/p} \rho_\lam(x) S_p(\varphi_\lam(x)) \\
    &\texttt{Let } \alpha= w(1)\\
    &\texttt{while }(|\alpha|\geq \tau) \\
    &\quad \quad \lam=(a+b)/2 \\
    &\quad \quad \texttt{Solve the ODE }\ref{ec_prufer} \texttt{ and obtain } \varphi_\lam \texttt{ and } \rho_\lam \\
    &\quad \quad \texttt{Let } w(x)=(p-1)^{1/p}\left(\lam r(x)\right)^{-1/p} \rho_\lam(x) S_p(\varphi_\lam(x)) \\
    &\quad \quad \texttt{Let } \beta= w(1) \\
    &\quad \quad \texttt{If } (\alpha \beta <0) \\
    &\quad \quad \quad \quad b=(a+b)/2 \\
    &\quad \quad \texttt{else} \\
    &\quad \quad \quad \quad a=(a+b)/2 \\
    &\texttt{end while} \\
    &\texttt{Then }\lam \texttt{ is the aproximation of eigenvalue with error } \leq \tau
\end{align*}

For example, let us consider $r(x)=2+\sin (2\pi x)$. In this case we obtain that
$\overline{r}=\int_I 2+\sin(2\pi x) dx =2$, and the eigenvalues of the limit problem
are given by
$$\lam_k^{1/p}=\frac{k\pi_p}{2^{1/p}}.$$ When $\ve$ tends to zero the value
of $\lam^\ve$ tends to the limit value $\lam$ displaying oscillations.

When $p=2$ the first limit eigenvalue is $\sqrt{\lam_1}=\pi/\sqrt{2}\sim 2.221441469$.
We see the oscillating behavior when plot $\sqrt{\lam_1^\ve}$ as function of $\ve$ in
Figure \ref{fig:p1}.
\begin{figure}[ht]
    \begin{center}
          \includegraphics{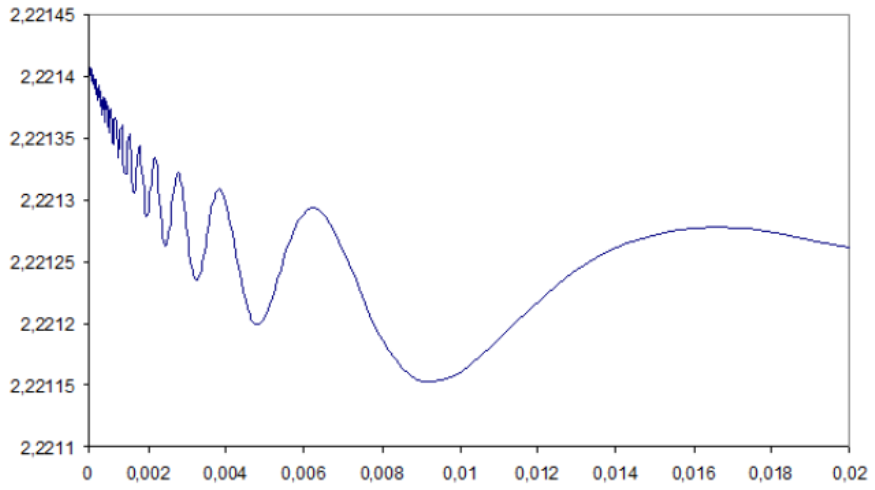}
    \end{center}
    \caption{The square root of the first eigenvalue as a function of $\ve$  when $r(x)=2+\sin (2\pi x)$.}
    \label{fig:p1}
\end{figure}

A more complex behavior can be found in Figure \ref{fig:p3}, where we considered the
weight $r(x)=\frac{1}{2+\sin{2\pi x}}$. We observe that the sequence tends to
$$\lam_1=\pi^2/\int_I \frac{1}{2+\sin{2\pi x}} dx = \sqrt{3} \pi \sim 17.09465627.$$
\begin{figure}[ht]
    \begin{center}
          \includegraphics[width=7.7cm]{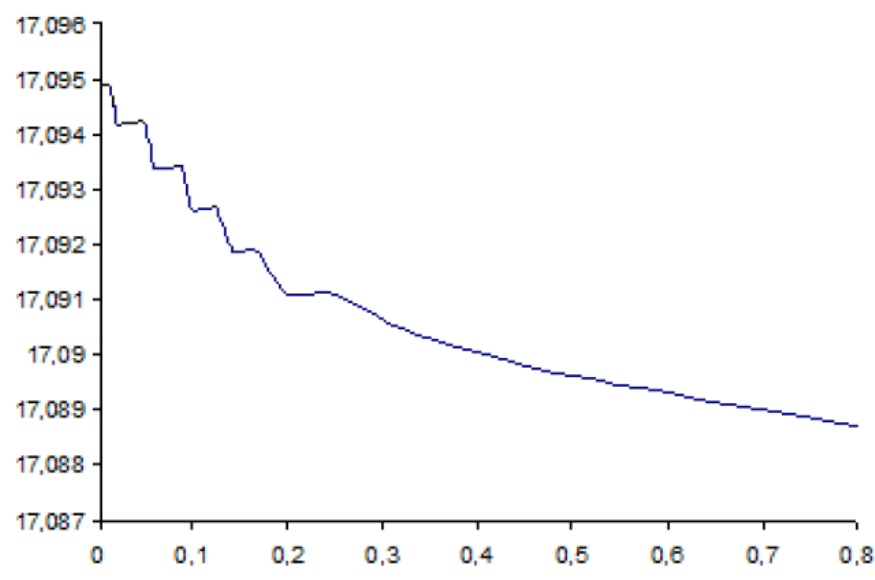}
    \end{center}
    \caption{The square root of the first eigenvalue as a function of $\ve$ when $r(x)=\frac{1}{2+\sin{2\pi x}}$.}
    \label{fig:p3}
\end{figure}

It is not clear why the convergence of the first eigenvalue display the oscillations
and the monotonicity observed (although the monotonicity is reversed for the weight
$r(x)=2-\sin{2\pi x}$). We believe that some Sturmian type comparison theorem with
integral inequalities for the weights (instead of point-wise inequalities as usual) is
involved. However, we are not able to prove it, and for higher eigenvalues it is not
clear what happens.

Turning now to the eigenfunctions, with the weight $r(x)=2+\sin (2\pi x)$,  the
normalized eigenfunction associated to the first eigenvalue of the limit problem is
given by $u_1(x)=\pi^{-1}\sin(\pi x)$. Applying the numerical algorithm we obtain that
the graph of an eigenfunction associated to the first eigenvalue $\lam_1^\ve$
intertwine with the graph of $u_1(x)$. When $\ve$ decreases, the number of crosses
increases, and  the amplitude of the difference between them decreases. In Figure
\ref{fig:p2} we can observe this behavior and the difference between $u_1$ and
$u_1^\ve$ for different values of $\ve$.

\begin{figure}[ht]
    \begin{center}
          \includegraphics[width=12cm]{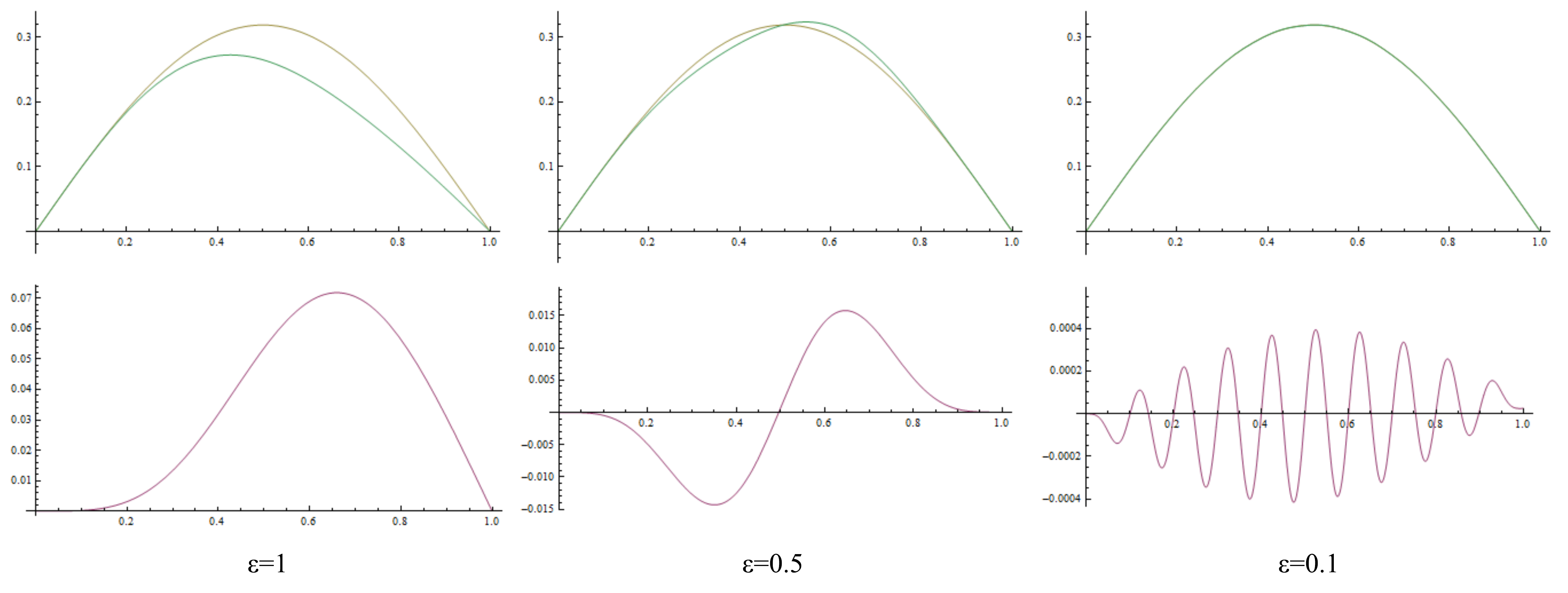}
    \end{center}
    \caption{The first eigenfunctions and the difference between them
    for different values of $\ve$.}
    \label{fig:p2}
\end{figure}

To our knowledge, it is not known any result about the number of the oscillations as
$\ve$ decreases, nor it is known if those oscillations disappear for $\ve$
sufficiently small.


The same behavior seems to hold for the higher eigenfunctions, see in Figure
\ref{fig:p4} the behavior of the fourth eigenfunction $u_4^\ve$ when the parameter
$\ve$ decrease.

\begin{figure}[ht]
    \begin{center}
          \includegraphics[width=12cm]{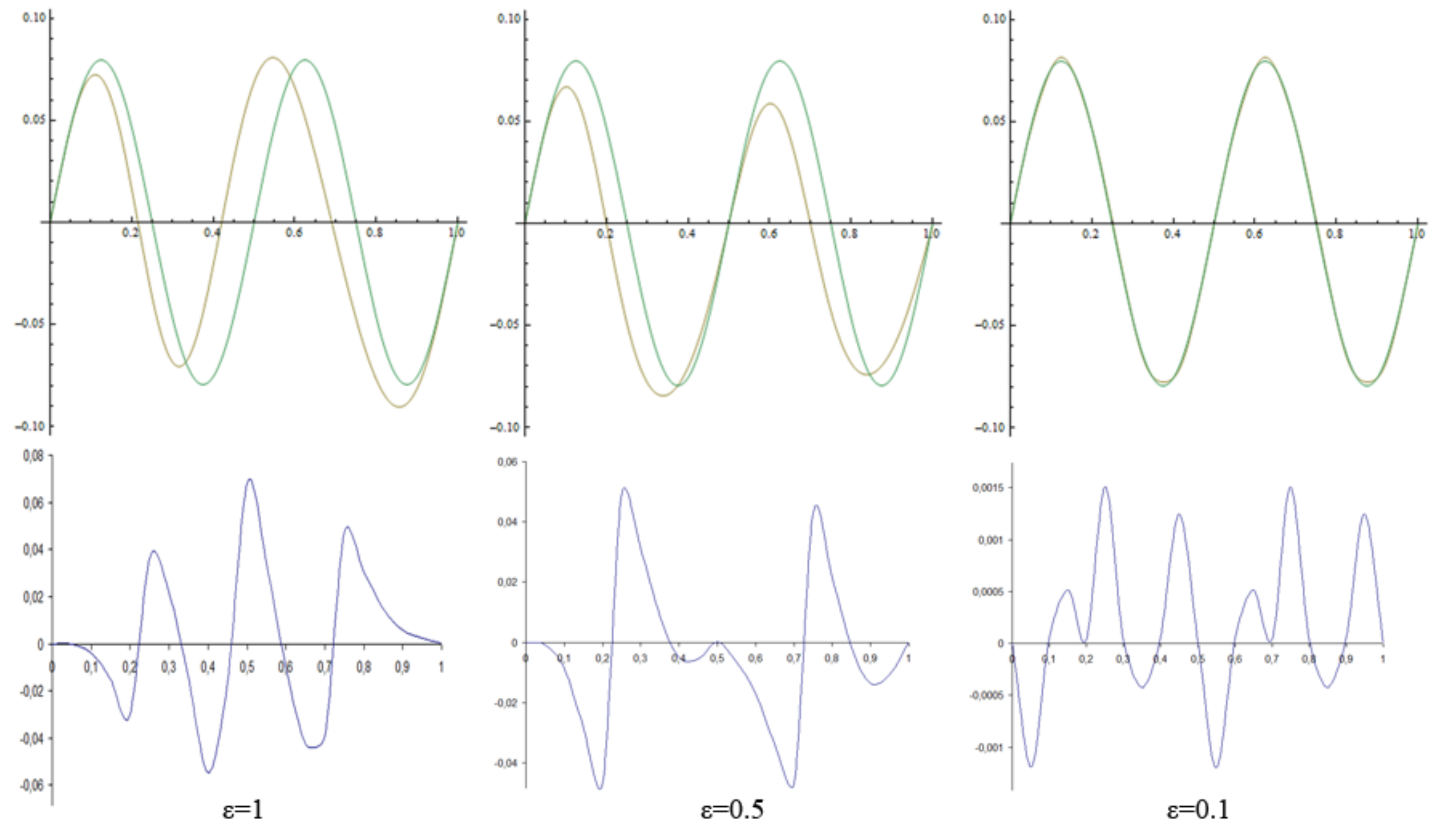}
    \end{center}
    \caption{The fourth eigenfunctions and the difference between them for different values of $\ve$.}
    \label{fig:p4}
\end{figure}

Here, the convergence of the nodal domains and the fact that the restriction of an
eigenfunction to one of its nodal domains $\mathcal{N}$ coincides with the first
eigenfunction of the problem in $\mathcal{N}$, together with the continuous dependence
of the eigenfunctions on the weight and the length of the domain, suggest that the
presence or not of oscillations for the higher eigenfunctions must be the same as for
the first one. However, the computations show very complex patterns in the
oscillations.

\section*{Acknowledgements}

This work was partially supported by Universidad de Buenos Aires under grant 20020100100400 and by CONICET (Argentina) PIP 5478/1438.

\def\cprime{$'$}
\providecommand{\bysame}{\leavevmode\hbox to3em{\hrulefill}\thinspace}
\providecommand{\MR}{\relax\ifhmode\unskip\space\fi MR }
\providecommand{\MRhref}[2]{%
  \href{http://www.ams.org/mathscinet-getitem?mr=#1}{#2}
}
\providecommand{\href}[2]{#2}

\end{document}